\documentclass[a4paper,10pt]{article}

\usepackage{amsmath}
\usepackage{amsthm}
\usepackage{amssymb}

\usepackage{indentfirst,amsmath,amsfonts,amstext,amssymb,amscd,bezier,amsthm}
\usepackage[latin1]{inputenc}
\usepackage[dvips]{graphicx}
\usepackage[latin1]{inputenc}
\usepackage{setspace}
\usepackage{graphicx}
\usepackage{tabularx}
\usepackage{longtable}
\usepackage[usenames,dvipsnames]{pstricks}
\usepackage{epsfig}
\usepackage{pst-grad}
\usepackage{pst-plot}
\usepackage[all]{xy}
\usepackage{hyperref}
\usepackage[T1]{fontenc}
\usepackage{ae,aecompl}
\usepackage{indentfirst}
\usepackage{xspace}
\usepackage{color}
\usepackage{psfrag}
\usepackage{mathrsfs}
\usepackage{upgreek}                            
\usepackage{makeidx}                            
\usepackage{fancyhdr}
\usepackage{fancybox}
\usepackage[rm]{titlesec}
\usepackage{float}   
\usepackage{enumerate}

\input xy
\xyoption{all}

\newcommand{\F}{\mathbb{F}}

\def \P {\mathbb{P}}

\newcommand{\ord}{\operatorname{ord}}

\theoremstyle{plain}
\newtheorem{thm}{Theorem}[section]
\newtheorem{defi}[thm]{Definition}

\newtheorem{lem}[thm]{Lemma}
\newtheorem{cor}[thm]{Corollary}

 \font\numberfont= pzcmi scaled
3000
\titleformat{\chapter}[display]
  {\normalfont\Large 
  }
  {
   \filright
   \rule[32pt]{.7\linewidth}{4pt}
   \hspace{-8pt}
   \shadowbox{
   \begin{minipage}{.15\linewidth}
     \begin{center}
          \textsl{\bf {\large \chaptertitlename}}\\
       \vspace{1ex}
       {\bf {\numberfont \thechapter}}\\
       \vspace{1ex}
     \end{center}
   \end{minipage}}
  }
  {-10pt}
  {\filcenter
           \sl
           \bf
              \Huge
     }
  [\vspace{-1cm}\hfill\rule{.8\textwidth}{0.5pt}\\
\vskip-2.8ex\hfill\rule{.7\textwidth}{4pt}\vspace*{-1ex}]
\titlespacing{\chapter}{0pt}{*4}{*1}

\titleformat{\section}[block]
{\normalfont\bfseries} {\thesection}{0.5em}{}

\titleformat{\subsection}[block]
{\normalfont\large\bfseries} {\thesubsection}{0.5em}{}

\makeindex                                      
\makeatletter
\numberwithin{equation}{section}

\begin{document}

\title{Points on singular Frobenius nonclassical curves}

\author{\textbf{Herivelto Borges} \\
 \small{ICMC, Universidade de S\~ao Paulo, S\~ao Carlos, Brazil} \\
 \textbf{Masaaki Homma}\\
 \small{Department of Mathematics and Physics, Kanagawa University, Hiratsuka 259-1293, Japan}}

\maketitle
 \begin{abstract}
In 1990, Hefez and Voloch proved that the number of $\F_q$-rational points on a nonsingular plane  $q$-Frobenius nonclassical curve of degree $d$ is $N=d(q-d+2)$.
 We address these curves in the  singular setting. In particular, we  prove that $d(q-d+2)$ is a lower  bound on the number of $\F_q$-rational points on such curves of degree $d$.
 \end{abstract}
 \smallskip
\noindent \text{Keywords:} Algebraic curve, Frobenius nonclassical curve, Finite Field.

\smallskip
\noindent \text{2010 Mathematics Subject Classification:} 
 Primary 14H45; Secondary 14Hxx.



\section{Introduction}\label{intro}

Let $p$ be a prime number and $\F_q$ be the field with $q=p^s$ elements, for some integer $s\geq 1$. An irreducible plane curve $\mathcal{C},$ defined over $\F_q$, is called $q$-Frobenius nonclassical if the $q$-Frobenius map takes each simple point $P\in \mathcal{C}$ to the tangent line to $\mathcal{C}$ at $P$. 
In this case, there is an exponent $h$ with
$p \leq p^h \leq d$ so that the intersection multiplicity
$i(\mathcal{C}.T_P(\mathcal{C});P)$ of $\mathcal{C}$ and the tangent 
line
$T_P(\mathcal{C})$ at a simple point $P \in \mathcal{C}$
is at least $p^h$, and actually $i(\mathcal{C}.T_P(\mathcal{C});P) = p^
h$
holds for a general point $P \in \mathcal{C}$.

For convenience,
\begin{equation}\label{eq:cases}
\nu =\begin{cases}
p^h \hspace{0.3cm} &  \text{ if  $\mathcal{C}$ is $q$-Frobenius 
nonclassical}\\
1 & \text{if  $\mathcal{C}$ is $q$-Frobenius classical}
\end{cases}
\end{equation}
is called the $q$-Frobenius order of $\mathcal{C}$.

Frobenius nonclassical curves were introduced in  the work of St\"ohr and Voloch  \cite{SV}, and 
one  reason for highlighting this special class of curves comes from the following result (see \cite[Theorem 2.3]{SV}).

\begin{thm}[St\"ohr-Voloch]\label{SV} Let $\mathcal{C}$ be an irreducible plane curve of degree $d$ and genus $g$ defined over $\F_q$. If  $\mathcal{C}(\F_q)$ denotes the set of $\F_{q}$-rational points on $\mathcal{C}$, then
\begin{equation}\label{SV-bound}
\#\mathcal{C}(\F_q)\le \dfrac{\nu (2g-2) +(q+2)d}{2}.
\end{equation}
\end{thm}
Note that by $\F_{q}$-rational points on $\mathcal{C}$, we mean the $\F_{q}$-rational points
on the nonsingular model of $\mathcal{C}$.
Based on Theorem \ref{SV}, Frobenius nonclassicality can be considered as an obstruction to 
use the nicer upper bound  given by inequality (\ref{SV-bound}) with $\nu=1$.
That is a clear reason why one should  try to understand such curves better. At the same time, investigating Frobenius nonclassical curves is a way of searching for curves with many points.
For instance,  the Hermitian curve 
$$x^{q+1}+y^{q+1}=1,$$
over $\F_{q^2}$, and the Deligne-Lusztig-Suzuki curve over $\F_q$:
$$y^q-y=x^{q_0}(x^q-x),$$
where $q_0=2^s$, $s\geq 1$, and $q=2q_0^2$,
which are well known examples of curves with many points, are Frobenius non-classical. 

With regard to the number of rational points, a somewhat surprising fact was proved by Hefez and Voloch in the case of nonsingular curves (see \cite{HV}).

\begin{thm}[Hefez-Voloch] Let $\mathcal{X}$ be a nonsingular $q$-Frobenius nonclassical  plane curve of degree $d$ defined over $\F_q$. If $\mathcal{X}(\F_q)$ denotes the set of $\F_{q}$-rational points on $\mathcal{X}$, then
\begin{equation}\label{HV-sharp}
\#\mathcal{X}(\F_q)=d(q-d+2).
\end{equation}
\end{thm}

Let us recall that  if  $\mathcal{X}$ is a nonsingular $q$-Frobenius nonclassical  plane curve of degree $d$, and   $\nu>2$ is its  $q$-Frobenius order  defined in \eqref{eq:cases}, then  (see \cite[Theorem 8.77]{HKT})
\begin{equation}\label{d-range}
\sqrt{q}+1\leq d\leq \frac{q-1}{\nu-1}.
\end{equation}
Now note that if $\nu>3$ and $d$ is within the range given by \eqref{d-range}, then
\begin{equation}\label{compare}
d(q-d+2)>\frac{d(q+d-1)}{2},
\end{equation}
where the number on the right hand side of \eqref{compare} is the bound  given by Theorem \ref{SV} for the case $\nu=1$.
In other words, \eqref{HV-sharp} tells us that nonsingular Frobenius nonclassical curves of degree $d$ usually have many rational points in comparison with the Frobenius classical ones of the same degree. In this paper, we show that this statement could be applied more broadly if we were to drop the exclusivity on nonsingularity. More precisely, we prove the following:

\begin{thm}\label{main0}Let  $\mathcal{C}$ be  a   $q$-Frobenius nonclassical curve of degree $d$ and genus $g$.  If $M_q^S$ is the number of simple points of $\mathcal{C}$ in 
$PG(2,q)$, then
\begin{equation}\label{bound1}
M_q^S\geq d(q-d+2)+2(g^*-g) +\sum\limits_{P \in Sing(\F_{q})} m_P(m_P-2),
\end{equation}
where $m_P$ are the multiplicities of the singular points $P\in Sing(\F_{q})\subseteq PG(2,q)$ of $\mathcal{C}$,
and $$g^*:=\frac{(d-1)(d-2)}{2}-\sum\limits_{P \in Sing(\F_{q})} \frac{1}{2}m_P(m_P-1)$$ is its $\F_q$-virtual genus. Moreover, equality holds in \eqref{bound1} if and only if all branches of  $\mathcal{C}$ are linear.
\end{thm}

Note that the  bound  \eqref{bound1} does not depend on the Frobenius order $\nu$. A very interesting consequence of Thorem \ref{main0} is the following:

\begin{cor}\label{main1}Let  $\mathcal{C}$ be  a   $q$-Frobenius nonclassical curve of degree $d$.  If $M_q$ is the number of points of $\mathcal{C}$ in 
$PG(2,q)$, then
\begin{equation}
M_q\geq d(q-d+2),
\end{equation}
 and equality holds if and only if $\mathcal{C}$ is nonsingular.
\end{cor}


\section{Preliminaries}\label{prelim}

Let us begin by briefly recalling the notions of 
classicality and $q$-Frobenius classicality for plane
curves. For a more general discussion, including the notion and properties of branches, we refer to \cite{HKT} and \cite{HK}.

 Let $\mathcal{C}\subset \P^2$ be an irreducible  algebraic curve of degree $d$ and genus $g$.
The numbers $0=\epsilon_0<\epsilon_1=1<\epsilon_2$
represent all possible intersection multiplicities of $\mathcal{C}$ with lines
of $\P^2$ at a generic point of $\mathcal{C}$. Such a sequence is called the order
sequence of $\mathcal{C}$, and it can be  characterized as the smallest sequence (in
lexicographic order) such that $\det(D_{\zeta}^{\epsilon_i}x_j)\ne0$, where
$D_\zeta^{k}$ denotes the $k$th Hasse derivative with respect to a separating
 variable $\zeta$,  and $x_0,x_1,x_2$ are the coordinate functions on
$\mathcal{C}\subset  \P^2$. The curve $\mathcal{C}$ is called classical if $\epsilon_2=2$.

If $\mathcal{C}$ is defined over a finite field $\F_q$, then there is a smallest
integer $\nu \in\{1,\epsilon_2\}$
 such that
 
\begin{equation}\label{det}
\begin{pmatrix}
  x_0^q & x_1^q &  x_2^q\\
  x_0& x_1  &  x_2 \\
  D_{\zeta}^{\nu} x_0& D_{\zeta}^{\nu}x_1  &
D_{\zeta}^{\nu}x_2
\end{pmatrix}
\neq 0
\end{equation}
The number $\nu$  is called the $q$-Frobenius order of $\mathcal{C}$, and such a
curve is called $q$-Frobenius classical if $\nu=1$.

Associated to the curve  $\mathcal{C}$,  there exist two distinguished divisors $R$ and $S$,
which play an important role in estimating the number of $\F_q$-rational points of $\mathcal{C}$.  When the curve is Frobenius nonclassical,
some valuable information can be  obtained by comparing  the multiplicities $v_P(R)$ and $v_P(S)$  for the points $P \in \mathcal{C}$.  In general, computing  these multiplicities is tantamount to studying some functions in $\overline{\F}_{q}(x,y)$ given by  Wronskian determinants such as $\det(D_{\zeta}^{\epsilon_i}x_j)$ and \eqref{det}. This idea was first exploited by Hefez and Voloch,  in their investigation of  the nonsingular case \cite{HV}. As noted by Hirschfeld and Korchm\'aros in \cite{HK}, this idea can be useful in the singular case as well.

Let  $\overline{\F}_{q}(\mathcal{C}):=\overline{\F}_{q}(x,y)$ be the function field of an irreducible curve $\mathcal{C}: f(x,y)=0$. Recall that  for any given place $\mathcal{P}$ of $\overline{\F}_{q}(\mathcal{C})$  and a local parameter $t$ at $\mathcal{P}$,  one  can associate   a (primitive)  branch  $\gamma$ in special affine coordinates:
$$x(t)=a+a_1t^{j_1}+\cdots,   \quad       y(t)=b+b_1t^{s}+\cdots,$$
where  $s\geq j_1$.  The point $(a,b) \in \overline{\F}_{q}\times \overline{\F}_{q}$ is called
the center of the branch $\gamma$.

The  branch $\gamma$ is called linear if  $j_1=1$. If  $p\nmid j_1$ (resp. $p\mid j_1$) then the branch is called  tame (resp. wild). Obviously, linear branches are tame.

When the curve $\mathcal{C}: f(x,y)=0$ is defined over $\F_{q}$, then $\mathcal{C}(\F_q)$ will   denote the set of places of degree one in the function field $\F_{q}(\mathcal{C})$. Considering  the projective closure $F(x,y,z)=0$ of $\mathcal{C}$, we define the following numbers, which are clearly related to $\#\mathcal{C}(\F_q)$:
\begin{defi}\label{define}
\begin{enumerate}[\rm(i)]
\item $M_q^S=$ number of  smooth points of $F(x,y,z)=0$ in $PG(2,q)$.
\item $M_q=$ number of   points of $F(x,y,z)=0$ in $PG(2,q)$.
\item $B_q=$ number of branches of $\mathcal{C}$ centered at a point in $PG(2,q)$. 
\end{enumerate}
\end{defi}
Note that  

\begin{equation}\label{ineq0}
M_q^S\leq M_q\leq B_q \text{ and }  M_q^S\leq \#\mathcal{C}(\F_q)\leq B_q.
\end{equation}
Hereafter, $\mathcal{C}$ will denote an irreducible plane curve of degree $d$ and genus $g$
defined over $\F_q$. A relevant step to prove our main result is based on the following:

\begin{thm}[Hirschfeld-Korchm\'aros]\label{thmHK} Assume that  $\mathcal{C}$ has only tame branches. If $\mathcal{C}$ is a nonclassical and  $q$-Frobenius nonclassical curve, then
$$B_q\geq (q-1)d-(2g-2),$$
and equality holds if and only if every singular branch of $\mathcal{C}$ is centered at a point of $PG(2,q)$.
\end{thm}

The next lemma  extends  Hirschfeld-Korchm\'aros'  result, and  our proof  is built on theirs. In particular, all the definitions and notations, explained in   detail in \cite{HK}, will be borrowed.

\begin{lem}\label{lemaHK}If  $\mathcal{C}$ is $q$-Frobenius nonclassical, then there exist at least $(q-1)d-(2g-2)$ tame branches centered at a point of $PG(2,q)$.
In particular, 
\begin{equation}\label{ineq1}
B_q\geq (q-1)d-(2g-2).
\end{equation}
Moreover, if every  branch  centered at a point of $PG(2,q)$ is tame, then \eqref{ineq1} is an equality  if and only if all the remaining branches are linear.

\end{lem}
\begin{proof} We closely follow the notation used in \cite{HK}.
\\
Set
\begin{center}
$
\det(D_{\zeta}^{(\epsilon_i)} x_j)=
\begin{vmatrix}
D_{\zeta}^{(1)} x & D_{\zeta}^{(1)} y  \\ 
D_{\zeta}^{(p^m)} x & D_{\zeta}^{(p^m)} y
\end{vmatrix}
$
\quad
and 
\quad
$
\det(D_{\zeta}^{(\nu_i)} x_j)=
\begin{vmatrix}
x^q-x & y^q-y  \\ 
D_{\zeta}^{(p^m)} x & D_{\zeta}^{(p^m)} y
\end{vmatrix}
$
\end{center}
The $q$-Frobenius nonclassicality of   $\mathcal{C}$ gives

\begin{equation}\label{fnc}
\begin{vmatrix}
x^q-x & y^q-y  \\ 
D_{\zeta}^{(1)} x & D_{\zeta}^{(1)} y
\end{vmatrix}
=0,
\end{equation}
 and then  establishes the relation 
$$\det(D_{\zeta}^{(\nu_i)} x_j)\cdot D_{\zeta}^{(1)} x=\det(D_{\zeta}^{(\epsilon_i)} x_j)\cdot (x^q-x).$$
Therefore, for any place $\mathcal{P}$ of  $\overline{\F}_{q}(\mathcal{C})$,
\begin{equation}\label{eq1}
{v}_{\mathcal{P}}(S)-{v}_{\mathcal{P}}(R)=\ord_{\mathcal{P}}(x^q-x)-\ord_{\mathcal{P}}(D_{\zeta}^{(1)} x).
\end{equation}
Let $\gamma$ be the (primitive) branch associated to the place $\mathcal{P}$,  represented by
$$x(t)=a+a_1t^{j_1}+\cdots,  \quad    y(t)=b+b_1t^{s}+\cdots,$$
with $j_1\leq s$. If $\gamma$ is tame, i.e., $p\nmid j_1$, then it follows (see  \cite[proof of  Theorem 1.4]{HK})  that

\begin{equation}\label{tame}
{v}_{\mathcal{P}}(S)-{v}_{\mathcal{P}}(R)= \begin{cases} 
1,  \quad \text{if } (a,b)\in \F_q\times \F_q;\\
-(j_1-1),  \quad \text{otherwise}.
\end{cases} 
  \end{equation}
  Now  let us address the wild case, i.e., the case $p\mid j_1$. Note that if $D_{\zeta}^{(1)} x=0$ then, from \eqref{fnc}, we have $D_{\zeta}^{(1)} y=0$,
which contradicts the primitivity of $\gamma$. Hence, $\ord_{\mathcal{P}}(D_{\zeta}^{(1)} x) = k> j_1-1$ and \eqref{eq1} yield

\begin{equation}\label{wild}
{v}_{\mathcal{P}}(S)-{v}_{\mathcal{P}}(R)= \begin{cases} 
-(k-j_1),  \quad \text{if } a\in \F_q;\\
-k,  \quad \text{otherwise}.
\end{cases} 
  \end{equation}
  
  Therefore,  \eqref{tame}  and \eqref{wild} can be reduced to
    
\begin{equation*}
{v}_{\mathcal{P}}(S)-{v}_{\mathcal{P}}(R)= \begin{cases} 
1,  \quad \text{if $\gamma$ is tame  with  center  in } PG(2,q);\\
\leq 0,  \quad \text{otherwise}.
\end{cases} 
  \end{equation*}
  
  Hence, since $\deg (S-R)=d(q-1)-(2g-2)$,  we arrive at the desired lower bound for the number 
  of tame branches centered at a point of $PG(2,q)$.
  
  Now let us assume that every  branch   centered at a point of $PG(2,q)$ is tame. If $B_q=d(q-1)-(2g-2)$, then \eqref{tame} implies that the remaining tame branches are
  linear. In addition,  \eqref{wild} implies that any wild branch can be considered as
     $$x(t)=a+a_1t^{j_1}+\cdots,  \quad    y(t)=b+b_1t^{s}+\cdots,$$
with $2\leq j_1\leq s$,  $\ord_{\mathcal{P}}(D_{\zeta}^{(1)} x)=j_1$ and $a\in \F_{q}$. 
However, if this is the case, then from 
$$\ord_{\mathcal{P}}\Big((x^q-x)(D_{\zeta}^{(1)} y)\Big)=\ord_{\mathcal{P}}\Big((y^q-y)(D_{\zeta}^{(1)} x)\Big),$$ we obtain 
$$\ord_{\mathcal{P}}(y^q-y)=\ord_{\mathcal{P}}(D_{\zeta}^{(1)} y)\geq s-1\geq 1,$$ i.e.,
  $b\in \F_{q}$. Thus, by hypothesis, such branch must be tame, and then the assertion follows.
  The converse follows immediately from the fact that linear branches are automatically tame.
 
\end{proof}

\section{The result}

The aim of  this section is to prove  Theorem \ref{main0} and some of its  relevant  corollaries.

\text{}\\

{ \bf Proof of Theorem \ref{main0}.} Note that from Lemma \ref{lemaHK}  and the definition of $B_q$, we have
\begin{equation}\label{ineq2}
(q-1)d-(2g-2)\leq B_q \leq  \sum\limits_{P \in PG(2,q)} m_P.
 \end{equation}
Let $M_q^S$ be the number of  smooth $\F_q$-points on $\mathcal{C}$, and set $g^*=\frac{(d-1)(d-2)}{2}-\sum\limits_{P \in Sing(\F_{q})} \frac{1}{2}m_P(m_P-1)$. Then 
 \begin{eqnarray*} 
M_q^S & = &\sum\limits_{P \in PG(2,q)} m_P-\sum\limits_{P \in Sing(\F_{q})} m_P\\
& = &\sum\limits_{P \in PG(2,q)} m_P-\sum\limits_{P \in Sing(\F_{q})} m_P(m_P-1)+\sum\limits_{P \in Sing(\F_{q})} m_P(m_P-2)\\
& = &\sum\limits_{P \in PG(2,q)} m_P+(2g^*-2)-(d^2-3d)+\sum\limits_{P \in Sing(\F_{q})} m_P(m_P-2)\\
& = &\sum\limits_{P \in PG(2,q)} m_P-\Big((q-1)d-(2g-2)\Big)+d(q-d+2)+2(g^*-g)+\sum\limits_{P \in Sing(\F_{q})} m_P(m_P-2).
\end{eqnarray*}
Since \eqref{ineq2} gives  $\sum\limits_{P \in PG(2,q)} m_P-\Big((q-1)d-(2g-2)\Big)\geq 0$,
it follows that
\begin{equation}\label{ineq3}
M_q^S\geq d(q-d+2)+2(g^*-g)+\sum\limits_{P \in Sing(\F_{q})} m_P(m_P-2).
\end{equation}
Now note that equality on the latter case is equivalent to equality on both sides of \eqref{ineq2}.
Let us assume we have equality in \eqref{ineq3}. The condition $B_q = \sum\limits_{P \in PG(2,q)} m_P$ means that all branches centered at a point of $PG(2,q)$ are linear and then tame.  Using the  additional equality $B_q=(q-1)d-(2g-2)$,   Lemma \ref{lemaHK} implies that all branches of  $\mathcal{C}$ are linear. Conversely, the linearity of all branches of  $\mathcal{C}$
immediately gives $B_q =\sum\limits_{P \in PG(2,q)} m_P$ and, from  Lemma \ref{lemaHK},
$B_q=(q-1)d-(2g-2)$.
\qed

\text{}\\

{ \bf Proof of Corollary \ref{main1}}.  From \eqref{ineq0} and \eqref{bound1}, we clearly have $M_q\geq d(q-2+d)$.
Let us assume that equality holds. Then  \eqref{ineq0}  and \eqref{bound1}  imply $M_q^S=M_q$ and
$g=g^*$,  respectively.  The first equality means that  all points $\F_q$-points of $\mathcal{C}$ are smooth, and thus $g^*=(d-1)(d-2)/2$. The latter equality, in addition, gives  $g=(d-1)(d-2)/2$. Therefore, $\mathcal{C}$ is a smooth curve. Conversely, if $\mathcal{C}$ is smooth then $M_q=B_q$, and Lemma \ref{lemaHK} gives $B_q=(q-1)d-(2g-2)$. Since $g=(d-1)(d-2)/2$, the result follows.

\qed

The following additional consequences are also worth mentioning. 

\begin{cor}\label{cor0} Let $\mathcal{C}$ be a $q$-Frobenius nonclassical curve of degree $d$ whose singularities are ordinary. If the singular points have their tangent lines  defined over $\mathbb{F}_q$, then 
 $$\#\mathcal{C}(\mathbb{F}_q)= d(q-d+2) +\sum\limits_{P \in \mathcal{C}} m_P(m_P-1).$$
\end{cor}

\begin{proof} Note that all singularities are ordinary and defined over $\mathbb{F}_q$. Thus 
$g^*=g$, and   equality in \eqref{bound1} holds.  That is,

$$M_q^S=d(q-d+2) +\sum\limits_{P \in Sing(\F_{q})} m_P(m_P-2).$$

On the other hand, since  the  tangent lines  of the singular points are defined over $\mathbb{F}_q$, each such point $P$ gives rise to  exactly $m_P$   $\mathbb{F}_q$-rational points of 
$\mathcal{C}$. Therefore 
$$\#\mathcal{C}(\mathbb{F}_q)=M_q^S+\sum\limits_{P \in Sing(\F_{q})} m_P=d(q-d+2) +\sum\limits_{P \in Sing(\F_{q})} m_P(m_P-1),$$ which gives the result.

\end{proof}

\begin{cor}\label{cor1} Let $\mathcal{C}$ be a $q$-Frobenius nonclassical curve of degree $d>1$.
Then $$d\geq \sqrt{q}+1,$$ and equality holds if and only if $\mathcal{C}$ is ($\F_{q}$-isomorphic to) the Hermitian curve.
\end{cor}
\begin{proof}
By Theorem \ref{main0} and Hasse-Weil bound, we have 
$$d(q-d+2)\leq M_q^S\leq 1+q+(d-1)(d-2)\sqrt{q}.$$
Since $d(q-d+2)\leq 1+q+(d-1)(d-2)\sqrt{q}$ if and only if
$(d-1)(\sqrt{q}+1)(\sqrt{q}+1-d)\leq 0$, the inequality $d\geq \sqrt{q}+1$ follows.
The additional assertion follows from a well known characterization of the Hermitian curve
(see e.g. \cite[Theorem 10.47]{HKT}).
\end{proof}

\begin{cor} Let $\mathcal{C}$ be a plane curve  defined over $\F_q$ of degree $d$, with $1<d\leq \sqrt{q}$, and genus $g$. Then

$$\#\mathcal{C}(\mathbb{F}_q) \leq \frac{(2g-2)+(q+2)d}{2}.$$
\end{cor}
\begin{proof}
This follows directly from Corollary \ref{cor1} and Theorem \ref{SV}.
\end{proof}

\section{Examples}

One can find several examples of Frobenius nonclassical curves that ilustrate the previous results
(see \cite{Bo} and \cite{Bo1}). Let us consider the particular curve
\begin{equation}\label{ex}
\mathcal{C}: x^4y^2 + x^2y^4 + x^4yz + xy^4z + x^4z^2 + x^2y^2z^2 + y^4z^2 +x^2z^4 + xyz^4 + y^2z^4=0
\end{equation}
over $\mathbb{F}_4$. This curve  has some  remarkable  properties (see \cite{Bo} and \cite{Homma}). One  particular feature of  $\mathcal{C}$ is its $4$-Frobenius nonclassicality.
The set of singular points of $\mathcal{C}$ is the whole of $PG(2,2)$, and  all such singularities are nodes whose tangent lines are defined over $\mathbb{F}_4$. Therefore, Corollary \ref{cor0} gives

$$\#\mathcal{C}(\mathbb{F}_4)=6(4-6+2)+7\cdot 2\cdot (2-1)=14.$$

The next example ilustrates how the  choice of singular  $q$-Frobenius nonclassical curves of degree $d$, over  nonsigular ones of  the same degree, can make a significant  difference with respect to the number of rational points. Consider the curves
$$\mathcal{C}_1: x^{13}=y^{13}+z^{13}$$
and
$$\mathcal{C}_2: x^{13}=y^{13}+y^9z^4+y^3z^{10}+yz^{12}+2z^{13},$$
over $\mathbb{F}_{27}$. They are both $27$-Frobenius nonclassical, and only $\mathcal{C}_1$ is smooth. One can check that $\#\mathcal{C}_1(\mathbb{F}_{27})=208$, whereas $\#\mathcal{C}_2(\mathbb{F}_{27})=280$, in addition to $\mathcal{C}_2$ being of smaller genus.


\printindex

\end{document}